\documentclass{amsart}
\usepackage{amssymb}
\newtheorem{theorem}{Theorem}[section]
\newtheorem{lemma}[theorem]{Lemma}
\newtheorem{proposition}[theorem]{Proposition}

\newtheorem{corollary}[theorem]{Corollary}

\theoremstyle{definition}

\theoremstyle{remark}
\newtheorem{remark}[theorem]{Remark}

\numberwithin{equation}{section}

\newcommand{\dens}{\mathrm{dens}}
\renewcommand{\span}{\mathrm{span}}
 
\newcommand{\dist}{\mathrm{dist}}
\newcommand{\conv}{\mathrm{conv}}

\newcommand{\R}{\mathbb{R}}

\newcommand{\X}{\mathrm{X}}
\renewcommand{\H}{\mathrm{H}}
\newcommand{\Y}{\mathrm{Y}}
\newcommand{\Z}{\mathrm{Z}}
\newcommand{\B}{\mathbf{B}}

\renewcommand{\S}{\mathbf{S}}

\renewcommand{\ker}{\mathrm{Ker}}
\renewcommand{\mod}{/}
\newcommand{\1}{\boldsymbol{1}}

\newcommand{\cf}{\mathrm{cf}}
\newcommand{\SCP}{\mathrm{SCP}}
\newcommand{\codens}{\mathrm{codens}}
\newcommand{\SD}{\mathrm{SD}}
\begin{document}

\title{Extracting long basic sequences from systems of dispersed vectors}

\begin{abstract}
We study Banach spaces satisfying some geometric or structural properties involving tightness of transfinite sequences of nested linear subspaces. These properties are much weaker than WCG and closely related to Corson's property $(C)$.
Given a transfinite sequence of normalized vectors, which is dispersed or null in some sense, we extract a subsequence which is a biorthogonal sequence, or even a weakly null monotone basic sequence, depending on the setting. 
The Separable Complementation Property is established for spaces with an M-basis under rather weak geometric properties. We also consider an analogy of the Baire Category Theorem for the lattice of closed linear subspaces.
\end{abstract}

\author{Jarno Talponen}
\address{Aalto University, Institute of Mathematics, P.O. Box 11100, FI-00076 Aalto, Finland} 
\email{talponen@cc.hut.fi}
\date{\today}
\maketitle

\section{Introduction}
Many Banach spaces have a rich structure of projections (see e.g. \cite{Kubis}). This can be applied in classifying spaces, and it is also often easier to work in separable complemented fragments of the space.
For example, spaces with a Schauder basis admit a very convenient structure of projections, especially if the basis is unconditional. On the other hand, spaces with very few projections (even few operators, see e.g. \cite{Ma}), like hereditarily idecomposable spaces, are currently an object of wide interest. This is partly due to several dichotomies about the existence of projections, roughly stating that if there are reasonably many projections, then there already exists a basic sequence with some required properties, like unconditionality (see e.g. \cite{AT2}, \cite{Gowers}).

This paper deals with nonseparable Banach spaces enjoying some structural properties much weaker than reflexivity. These
properties involve tightness conditions for transfinite chains of nested linear subspaces. We aim to show that 
spaces with such chains admit plenty of bounded linear projections. As a result, we will establish the separable complementation property, the existence of basic sequences, or other such properties depending on the setting. 
As the title of the paper suggests, the main problem here is to extract a transfinite basic sequence or a biorthogonal sequence from a net of vectors, which is in some sense far from being constant.

The following theorem is typical here and it is a kind of prototypical consequence of the main results.
\begin{theorem}\label{thm: first_main_0}
Let $\X$ be a Banach space and $\{x_{\alpha}\}_{\alpha<\kappa}\subset \X$ a normalized sequence, where $\kappa$ is an uncountable regular cardinal. 
\begin{enumerate}
\item[(1)]{If $\{x_{\alpha}\}_{\alpha<\kappa}$ is weakly null, then there is a subsequence $\{x_{\alpha_{\sigma}}\}_{\sigma<\kappa}$ which forms a monotone basic sequence.}
\item[(2)]{If $\X$ has Corson's property $(\mathrm{C})$ and $\{x_{\alpha}\}_{\alpha<\kappa}$ is dispersed (resp. strongly dispersed), then there is a subsequence $\{x_{\alpha_{\sigma}}\}_{\sigma<\kappa}$ which forms a bounded biorthogonal sequence 
(resp. a monotone basic sequence).}
\end{enumerate}
Moreover, each nonseparable Plichko space contains an uncountable monotone basic sequence.
\end{theorem}
This result will be given in more generality, and we will shortly provide the definition of dispersed and strongly dispersed sequences. Actually, the principle behind the above statement $(1)$ has been essentially known since the work of Bessaga-Pelczynski in 1958 (\cite{BP}), at least in the countable setting, and it is a natural example of the type of phenomena studied here. We will also study, motivated by the constructions of the projections, the intersections of a certain kind of fat subspaces.

Recall that the unit ball of a reflexive Banach space is weakly compact, which implies that each normalized sequence has a cluster point with respect to the weak topology. A similar conclusion holds for sequences of length $\omega_{1}$ 
should the unit ball be weakly Lindel\"{o}f. Thus, above we increased the length of the sequence while weakening the hypothesis. The possibility for this kind of trade-off is typical for Banach spaces. This is mostly due to the fact that Banach spaces are always countably 
tight in the norm and weak topologies, and more importantly, their \emph{duals} are often countably tight in the $\omega^{\ast}$-topology, or usually at least something similar holds. So, even though the weak clustering of \emph{countable} sequences is not necessarily common 
in general Banach spaces, the weak clustering of \emph{uncountable} sequences (of vectors, subspaces, etc.) is typical and can be ensured by imposing rather weak geometric or structural conditions. Thus the applications of countable and uncountable combinatorics differ considerably in Banach spaces. 

The following hypothesis for Banach spaces $\X$ becomes useful here:
\begin{itemize}
\item[$\bullet$]{For any uncountable, regular cardinal $\kappa$ each nested sequence $\{A_{\alpha}\}_{\alpha<\kappa}$ of closed affine subspaces of $\X$ has non-empty intersection.}
\end{itemize}
No infinite-dimensional space satisfies the above condition if we consider countable sequences instead. However, the above hypothesis follows readily if the space in question is Lindel\"{o}f in the weak topology, e.g. a WCG space.
Some other, related tightness conditions for subspace chains will be considered as well. Interestingly, in the last section it turns out that these conditions can be reformulated as properties of the canonical inclusion mappings 
$\X\hookrightarrow \underset{\longleftarrow}{\lim}\ \X\mod Y_{\alpha}$, where the range is an inverse limit corresponding to the chain $\{A_{\alpha}\}_{\alpha<\kappa}$.

In the light of some recent work, (e.g. \cite{DLT}, \cite{Ko}, \cite{To}), it is perhaps surprising that we do not require any additional set theoretic axioms to accomplish the main results. 
Namely, we will apply, in addition to Banach space theory, some fairly elementary techniques of infinitary combinatorics in topology.

Various types of problems regarding subsequences of \emph{countable} weakly null sequences have been studied previously in abundance (see e.g. \cite{AG}, \cite{LoTo}, \cite{Ray}, \cite{Ros}), but the combinatorial techniques and the conclusions appear to be very different in the countable vs. uncountable settings.

\newpage
\subsection{Preliminaries}
\ \\

Real, infinite-dimensional Banach spaces are typically denoted by $\X,\Y,\Z$. 
We denote by $\B_{\X}$ the closed unit ball of $\X$ and by $\S_{\X}$ the unit sphere of $\X$. 

Unless otherwise stated, $\lambda$ is a limit ordinal, $\kappa$ is an uncountable regular cardinal, and we will apply cardinal arithmetic (instead of ordinal arithmetic) notations. A subsequence $\{\alpha_{\sigma}\}_{\sigma<\beta}$ of $\lambda$
is a mapping $\beta \to \lambda,\ \sigma\mapsto \alpha_\sigma$, and, unless otherwise stated, subsequences are strictly increasing here.

See \cite{HHZ}, \cite{En} and \cite{Ne} for the standard notions in Banach spaces, set theory and topology, respectively.
We refer to Zizler's survey \cite{Z} on the nonseparable Banach spaces for most of the definitions and results used here.
The references \cite{Gr}, \cite{Hajek_biortsyst}, \cite{Sh}, \cite{SingerI}, \cite{SingerII}, \cite{To} provide suitable backround information on different kinds of biorthogonal and related systems.
See also the articles \cite{Do}, \cite{Kalenda2000}, \cite{La} and \cite{Sh2} related to the topics in this paper.

Denote $\dist(A,B)=\inf_{a\in A, b\in B}\|a-b\|$, where $A,B\subset \X$. 
We call a subset $A\subset\X$ \emph{generating} if $[A]$, the closed linear span of $A$, is $\X$.
Recall that $\mathcal{F}\subset \X^{\ast}$ is a separating subset, if and only if for each $x\in\X$ there is $f\in\mathcal{F}$
such that $f(x)\neq 0$, if and only if $\overline{[\mathcal{F}]}^{\omega^{\ast}}=\X^{\ast}$.
We say that a Banach space $\X$ has the \emph{Countable Separation Property} (CSP), if each separating subset 
$\mathcal{F}\subset\X^{\ast}$ contains a countable separating subset $\mathcal{F}_{0}\subset\mathcal{F}$.
Suppose that $\{(x_{\alpha},x^{\ast}_{\alpha})\}_{\alpha<\lambda}\subset \X\times \X^{\ast}$ 
is a biorthogonal system, i.e. $x_{\alpha}^{\ast}(x_{\beta})=\delta_{\alpha,\beta}$. We call $\{x_{\alpha}\}_{\alpha<\lambda}$ a \emph{biorthogonal sequence} for brevity. Equivalently, $\{x_{\alpha}\}_{\alpha<\lambda}$ is \emph{minimal}, that is, $x_{\beta}\notin [x_{\alpha}:\ \alpha\neq \beta]$ for all $\beta$. The latter concept is more operational for topological vector spaces. A biorthogonal system $\{(x_{\alpha},x^{\ast}_{\alpha})\}_{\alpha}$ is \emph{bounded}, if $\sup_{\alpha}\|x_{\alpha}\| \cdot \|x_{\alpha}^{\ast}\|<\infty$. The analogous minimality notion is the following: $\{x_{\alpha}\}_{\alpha<\lambda}$ is 
\emph{uniformly minimal}, if there is $C>0$ such that $\dist(x_{\beta},[x_{\alpha}:\ \alpha\neq \beta])>C$ for $\beta<\lambda$.
If $[x_{\alpha}: \alpha<\lambda]=\X$ and $\overline{[x^{\ast}_{\alpha}:\ \alpha<\lambda]}^{\omega^{\ast}}=\X^{\ast}$, then $\{(x_{\alpha},x^{\ast}_{\alpha})\}_{\alpha<\lambda}$ is called a \emph{Markusevic basis} or \emph{M-basis}. 
If $\{(x_{\alpha},x^{\ast}_{\alpha})\}_{\alpha<\lambda}$ is an M-basis on $[x_{\alpha}: \alpha<\lambda]$, then $\{x_{\alpha}\}_{\alpha<\lambda}$ is called an M-basic sequence.
We often do not explicitly include the functionals in M-basis.

Recall that a Banach space $\X$ has the \emph{Separable Complementation Property} (SCP) (resp. $1$-SCP), if each separable subspace is contained in a complemented (resp. $1$-complemented) subspace of $\X$.
A Banach space is \emph{Hereditarily Indecomposable} ($\mathrm{HI}$), if no subspace can be written as a direct sum of two of its infinite-dimensional closed subspaces. See e.g. \cite{AT} and \cite{AT2} for dicussion.

Given a limit ordinal $\theta$, a sequence $\{x_{\alpha}\}_{\alpha<\theta}\subset\X \setminus \{0\}$ is said to be a (transfinite) basic sequence, if there is a constant $1\leq C<\infty$ such that $\|y\| \leq C \|y+z\|$
for all $y\in [x_{\alpha}:\ \alpha<\lambda]$, $z\in [x_{\alpha}:\ \lambda\leq \alpha<\theta]$ and all $\lambda<\theta$ (see \cite[p.589]{SingerII}). In such case there are natural linear basis projections 
$P_{\lambda}\colon [x_{\alpha}:\ \alpha<\theta]\to [x_{\alpha}:\ \alpha<\lambda]$ such that $\|P_{\lambda}\|\leq C$ for $\lambda<\theta$. The basic sequence is said to be \emph{monotone}, if the basis projections are contractive, i.e. $C=1$.

Recall that a compact Hausdorff space $K$ is called a \emph{Corson compact}, if it can be embedded in a $\Sigma$-product of real lines, and a Banach space $\X$ is Weakly Lindel\"{o}f Determined (WLD), if $(\B_{\X^{\ast}},\omega^{\ast})$ is a Corson compact. 
The following condition is equivalent to the fact that $\X$ is WLD 
(see e.g. \cite[Thm.4.17]{Kalenda_Extracta2000}): There is an M-basis $\{x_{\alpha}\}_{\alpha}$ of $\X$ such that
\begin{equation}\label{eq: count}
|\{\alpha :\ f(x_{\alpha})\neq 0\}|\leq \aleph_{0}\quad \mathrm{for\ any}\ f\in \X^{\ast}. 
\end{equation}
A subspace $\Y\subset\X^{\ast}$ is called $r$-norming, $0< r\leq 1$, if $\inf_{x\in\S_{\X}}\sup_{x^{\ast}\in \S_{\Y}}x^{\ast}(x)\geq r$, or, equivalently, $r\B_{\X^{\ast}}\subset \overline{\B_{\Y}}^{\omega^{\ast}}$.
The space $\Y$ is norming, if it is $r$-norming for some $r>0$.
A Banach space is called \emph{Plichko}, if it admits a \emph{countably $1$-norming M-basis}, that is, an M-basis $\{x_{\alpha}\}_{\alpha\in \Gamma}$ such that 
$\{f\in \X^{\ast}:\ |\gamma:\ f(x_{\gamma})\neq 0|\leq \aleph_{0}\ \}\subset \X^{\ast}$ is a $1$-norming subspace.

Recall that a topological space $(T,\tau)$ is \emph{countably tight}, if for any $A\subset T$ and $x\in \overline{A}^{\tau}$ there is a countable subset $A_{0}\subset A$
such that $x\in \overline{A_{0}}^{\tau}$. The following folklore facts will be applied frequently, sometimes implicitly. 
Suppose that $(T,\tau)$ is a countably tight topological space and $\{E_{\alpha}\}_{\alpha<\kappa}$ is a family of closed subsets of $T$ such that $E_{\alpha}\subset E_{\beta}$ for $\alpha<\beta<\kappa$. 
Then $\overline{\bigcup_{\alpha<\kappa}E_{\alpha}}^{\tau}=\bigcup_{\alpha<\kappa}E_{\alpha}$. We will study several properties of Banach spaces $\X$ that are weakenings of $\omega^{\ast}$-countable tightness of the dual space $\X^{\ast}$.

Given a limit ordinal $\lambda$, we say that a sequence $\{x_{\alpha}\}_{\alpha<\lambda}\subset \X$ is \emph{dispersed} if 
\[\bigcap_{\gamma<\lambda}[x_{\alpha}:\ \gamma<\alpha<\lambda] \subsetneq [x_{\alpha}:\ \beta<\alpha<\lambda] \]
holds for all $\beta<\lambda$ and \emph{strongly dispersed} (SD) if it is dispersed and 
\[\bigcap_{\beta<\lambda}[x_{\alpha}:\ \beta<\alpha<\lambda]=\{0\}.\]
For example, any biorthogonal sequence (resp. M-basic sequence) is a dispersed (resp. $\SD$) sequence under any re-ordering. Here we are particularly interested in spaces admitting a normalized dispersed sequence with $\lambda$ being a regular uncountable cardinal. Nonseparable spaces not admitting any such $\SD$ sequence have been studied in \cite{Tal}.

Note that according to the Hahn-Banach theorem any weakly null sequence $\{x_{\alpha}\}_{\alpha<\kappa}\subset\X\setminus \{0\}$ is $\SD$. More generally, if $\Z\subset\X^{\ast}$ is a separating subspace, then each $\sigma(\X,\Z)$-null sequence of length $\kappa$ is $\SD$. If the Lindel\"{o}f number of $\X$ in its weak topology is less than $\kappa$, then also the converse holds, see Proposition \ref{prop: linde}. 
The sequences $\{e_{\alpha}\}_{\alpha<\lambda}\subset\ell^{1}(\lambda)$, $\{\1_{(\alpha,\lambda]}\}_{\alpha<\lambda}\subset C([0,\lambda])$ and $\{\1_{[\alpha,\lambda)}\}_{\alpha<\lambda}\subset \ell^{\infty}(\lambda)$ provide examples of $\SD$ sequences, 
which are not weakly null. Next, we will give as an example two important classes of spaces where $\SD$ sequences behave particularly well.
\begin{remark}
\

\begin{enumerate}
\item[$(1)$]{Let $\H$ be a Hilbert space and $\{x_{\alpha}\}_{\alpha<\kappa}\subset \H$ a normalized $\SD$ sequence. Then there is an orthonormal subsequence $\{x_{\alpha_{\sigma}}\}_{\sigma<\kappa}$.}
\item[$(2)$]{Let $\X$ be a weakly Lindel\"{o}f Banach space with an unconditional basis and let $\{x_{\alpha}\}_{\alpha<\kappa}\subset \X$ be a $\SD$ sequence. Then there is an unconditional basic sequence $\{x_{\alpha_{\sigma}}\}_{\sigma<\kappa}$.}
\end{enumerate}
The verification of these claims is easy after the considerations in Section 3.
\end{remark}

\section{Structural considerations}

\ \\
In this section we will study properties of Banach spaces, much weaker than reflexivity, related to chains of subspaces.

\subsection{Banach spaces enjoying good uncountable asymptotics}
\ \\
Next, we will enumerate some conditions involving the structure of Banach spaces $\X$ that subsequently turn out to be useful. These conditions are kind of subspace counterparts for the $\omega^{\ast}$-countable tightness of the dual space, 
or the weaker property $(C)$. Here $\kappa$ will stand for any uncountable regular cardinal, and $1\leq r <\infty$ is a real number.

\begin{enumerate}
\item[$(C)$]{$\X$ is said to have property $(C)$ (after H. Corson \cite{Corson}), if each family of closed convex sets of $\X$ with empty intersection has a countable subfamily with empty intersection.}
\item[$(C')$]{An equivalent reformulation of property $(C)$ (proved by R. Pol \cite{Pol}): given a set
$A\subset\X^{\ast}$ and $f\in \overline{A}^{\omega^{\ast}}$, there is a countable subset $A_{0}\subset A$ such that 
$f\in \overline{\conv}^{\omega^{\ast}}(A_{0})$.}
\item[$(\mathrm{I})$]{Each nested sequence of closed affine subspaces $\{A_{\alpha}\}_{\alpha<\kappa}$ of $\X$ has non-empty intersection.}
\item[$(\mathrm{II})$]{Given a nested sequence $\{Z_{\alpha}\}_{\alpha<\kappa}$ of closed linear subspaces of $\X$, there is for each $f\in \left(\bigcap_{\alpha<\kappa}Z_{\alpha}\right)^{\bot}$ an ordinal $\alpha<\kappa$ such that $f\in Z_{\alpha}^{\bot}$.}
\item[$(\mathcal{B})$]{Let $\{Z_{\alpha}\}_{\alpha<\kappa}$ be a nested sequence of closed linear subspaces of $\X$ such that
$\bigcap_{\alpha<\kappa}\Z_{\alpha}=\{0\}$. Then $\bigcap_{\alpha<\kappa}\overline{\B_{\X}+Z_{\alpha}}$ is bounded. (Considered in \cite{Tal}.)}
\item[$(r$-$\mathcal{B})$]{For $\{Z_{\alpha}\}_{\alpha<\kappa}$ as in $(\mathcal{B})$
it holds that $\bigcap_{\alpha<\kappa}\overline{\B_{\X}+Z_{\alpha}}\subset r\B_{\X}$.}
\end{enumerate}

Recall that we have the following implications: WCG $\implies$ WLD $\implies$ weakly Lindel\"{o}f $\implies$ property $(C)$. It is easy to see that $(C)\implies (\mathrm{I})$, $(C')\implies (\mathrm{II})$ and that the condition in $(\mathcal{B})$ (resp. in $(r$-$\mathcal{B}))$ holds, if and only if $\Y=\bigcup_{\alpha<\kappa}Z_{\alpha}^{\bot}\subset\X^{\ast}$ is a norming (resp. $1/r$-norming) subspace.

When applying these conditions, we use one sequence of subspaces at the time. So, it is not really essential here, if the respective conclusions are valid simultaneously for every sequence (and actually we could have formulated the conditions
$(\mathrm{I})$, $(\mathrm{II})$, $(\mathcal{B})$ specific to a given sequence). In many cases there are spaces 
failing the conditions, while there is a specific chain of subspaces, which clearly satisfies the respective condition.
These concepts are also discussed in the last section in terms of inverse limits. Next, we will briefly examine some examples.

Let us consider the space $\ell^{\infty}_{c}(\omega_{1})$ of countably supported vectors in $\ell^{\infty}(\omega_{1})$, and we denote by $Z_{\alpha},\ \alpha<\omega_{1},$ the subspace, where the first $\alpha$ coordinates vanish. Then $\bigcap_{\alpha<\omega_{1}}\B_{\ell^{\infty}_{c}(\omega_{1})}+Z_{\alpha}=\B_{\ell^{\infty}_{c}(\omega_{1})}$ (as in $(1$-$\mathcal{B})$), and it can be seen fairly easily that for each $f\in (\ell^{\infty}_{c}(\omega_{1}))^{\ast}$ there is $\alpha<\omega_{1}$ such that $f\in Z_{\alpha}^{\bot}$ (as in $(\mathrm{II})$). However, putting $\{\1_{[0,\alpha]}+Z_{\alpha}\}_{\alpha<\omega_{1}}$ defines a nested sequence of affine subspaces in $\ell^{\infty}(\omega_{1})$, 
whose intersection is the singleton $(1,1,1,\ldots)\in \ell^{\infty}(\omega_{1})$. This vector clearly escapes $\ell^{\infty}_{c}(\omega_{1})$, so that condition $(\mathrm{I})$ fails. 

On the other hand, suppose that $Y_{\alpha}\subset \ell^{\infty}(\omega_{1}),\ \alpha<\omega_{1},$ are the subspaces 
supported on $(\alpha,\omega_{1})$. Then $\bigcap_{\alpha<\omega_{1}}\1_{[0,\alpha]}+Y_{\alpha}=\{(1,1,1,\ldots)\}$ (as in $(\mathrm{I})$), 
$\bigcup_{\alpha<\omega_{1}}Y_{\alpha}^{\bot}\neq \ell^{\infty}(\omega_{1})^{\ast}$ ($(\mathrm{II})$ fails), $\bigcap_{\alpha<\omega_{1}} \B_{\ell^{\infty}(\omega_{1})}+Y_{\alpha}=\B_{\ell^{\infty}(\omega_{1})}$ (as in $(1$-$\mathcal{B})$).

An example of a space which does not satisfy $(\mathcal{B})$ has been constructed in \cite{Tal}, and the space $\ell^{1}(\omega_{1})$ is a more simple example. Namely, for each $\alpha=\omega\cdot \gamma + n <\omega_{1}$, 
where we apply ordinal arithmetic, $\gamma<\omega_{1}$ and $0<n<\omega$ is regarded as a real number, 
let $x_{\alpha}=n e_{n}+ (1 / n) e_{\alpha}\in \ell^{1}(\omega_{1})$ and $x_{\alpha}=0$ for limits $\alpha$. 
Then $n e_{n}\in [x_{\alpha}:\ \delta<\alpha<\omega_{1}]+(1/n)\B_{\ell^{1}(\omega_{1})}$ for each $\delta<\omega_{1}$ and $0<n<\omega$. 
Thus $n e_{n}\in \bigcap_{\delta<\omega_{1}}[x_{\alpha}:\ \delta<\alpha<\omega_{1}]+(1/m)\B_{\ell^{1}(\omega_{1})}$ for $n\geq m$.
By modifying this example one can find a chain of subspaces of $\ell^{1}(\omega_{1})$ corresponding to condition $(r$-$\mathcal{B})$ sharp for any given $r\geq 1$.

\begin{proposition}\label{prop1}
\

\begin{enumerate}
\item[(i)]{Condition $(\mathrm{I})$ of $\X$ yields that each quotient and subspace of $\X$ satisfies $(\mathcal{B})$.} 
\item[(ii)]{Condition $(\mathrm{II})$ of $\X$ yields that each quotient and subspace of $\X$ satisfies $(1$-$\mathcal{B})$.}
\item[(iii)]{The Countable Separation Property implies $(\mathrm{I})$ and $(1$-$\mathcal{B})$.}
\item[(iv)]{Suppose that $\X$ satisfies $(\mathrm{II})$. Then $\{x_{\alpha}\}_{\alpha<\kappa}\subset\X\setminus \{0\}$ is $\SD$ if and only if it is weakly null if and only if it is $\sigma(\X,\Z)$-null for any separating subspace $\Z\subset\X^{\ast}$.}
\end{enumerate}

\end{proposition}
\begin{proof}
To check (i) and (ii), it is easy to check that $(\mathrm{I})$ and $(\mathrm{II})$ pass on to quotients and subspaces. The first two claims follow directly from Theorem 3.11 and the proof of Proposition 3.9 in \cite{Tal}. 

To check (iii), it follows from the considerations in \cite{Tal} that, if $\X$ has the countable separation property, then for each nested sequence of closed subspaces $\{Z_{\alpha}\}_{\alpha<\kappa}$ with $\bigcap_{\alpha<\kappa}Z_{\alpha}=\{0\}$, 
it follows that $Z_{\alpha}=\{0\}$ already for some $\alpha<\kappa$.

In (iv) we begin by checking that if $\{x_{\alpha}\}_{\alpha<\kappa}$ is $\sigma(\X,\Z)$-null, then it is $\SD$. Suppose that $x\in \bigcap_{\beta<\kappa}[x_{\alpha}:\ \beta<\alpha<\kappa],\ x\neq 0$. 
Then there is $f\in \Z$ such that $f(x)\neq 0$. On the other hand, if $\{x_{\alpha}\}_{\alpha<\kappa}$ is $\sigma(\X,\Z)$-null 
then, due to reularity of $\kappa$, there is $\beta<\kappa$ such that $f(x_{\alpha})=0$ for $\beta<\alpha<\kappa$.
This contradicts the choice of $f$ and $x$.
 
To check that $\SD$ implies weakly null, let $\{x_{\alpha}\}_{\alpha<\kappa}\subset\X$ be $\SD$ and fix $f\in \X^{\ast}$.
Since $\left(\bigcap_{\beta<\kappa}[x_{\alpha}:\ \beta<\alpha<\kappa]\right)^{\bot}=\X^{\ast}$, we obtain by using $(\mathrm{II})$ that there is $\beta<\kappa$ with
$f\in [x_{\alpha}:\ \beta<\alpha<\kappa]^{\bot}$. This reads $f(x_{\alpha})=0$ for $\alpha\in (\beta , \kappa)$.

\end{proof}

\begin{corollary}
Let $\X$ be a Banach space with Corson's property $(C)$. If a sequence $\{x_{\alpha}\}_{\alpha<\kappa}\subset\X$ is $\sigma(\X,\Z)$-null for some separating subspace $\Z\subset\X^{\ast}$, 
then it is already weakly null.
\end{corollary}
\qed

\begin{proposition}\label{prop: linde}
Let $\X$ be a Banach space and assume that the Lindel\"{o}f number of $\X$ in its weak topology is less than $\kappa$. Then each $\SD$ sequence of length $\kappa$ is weakly null. 
\end{proposition}
\begin{proof}
Let $\{x_{\alpha}\}_{\alpha<\kappa}\subset \X$ be a $\SD$ sequence. Assume to the contrary that it is not weakly null.
Then there exists by the regularity of $\kappa$ a functional $f\in \X^{\ast}$, a constant $c>0$ and a subsequence $\{x_{\alpha_{\beta}}\}_{\beta<\kappa}$ such that $f(x_{\alpha_{\beta}})\geq c$ for $\beta<\kappa$.
By using the assumption about the Lindel\"{o}f number of $\X$ we get that $\{x_{\alpha_{\beta}}\}_{\beta<\kappa}$
has a weak cluster point $x\in \X$. It follows that $x\in \bigcap_{\gamma<\kappa}[x_{\alpha}:\ \gamma<\alpha<\kappa]$.
Necessarily $f(x)\geq c$, in particular $x\neq 0$, and thus we have a contradiction. 
\end{proof}

\begin{lemma}\label{lm: angle}
Let $\X$ be a Banach space, $\Y\subset\X$ a closed subspace with $\dens(\Y)<\kappa$, $\kappa$ an uncountable regular cardinal, and let $\{Z_{\alpha}\}_{\alpha<\kappa}$ be a nested sequence of closed subspaces of $\X$ with trivial intersection.
Suppose that 
\[\bigcap_{\alpha<\kappa}\overline{\B_{\X}+Z_{\alpha}}\subset r \B_{\X}\]
for some $1\leq r<\infty$. Then there exists $\beta<\kappa$ such that the angle between $\Y$ and $Z_{\beta}$, $\dist(\S_{\Y},Z_{\beta})\geq 1/r$.
\end{lemma}
\begin{proof}
First, observe that according to the assumption  
\begin{eqnarray*}
\bigcap_{\alpha<\kappa}\overline{(1-\epsilon)r^{-1}\B_{\X}+Z_{\alpha}}&=&\bigcap_{\alpha<\kappa}\overline{(1-\epsilon)r^{-1}\B_{\X}+(1-\epsilon)r^{-1}Z_{\alpha}}\\
&=&(1-\epsilon)r^{-1}\bigcap_{\alpha<\kappa}\overline{\B_{\X}+Z_{\alpha}}\subset(1-\epsilon)\B_{\X},
\end{eqnarray*}
for $0\leq\epsilon\leq 1$. Assume to the contrary that 
$\lim_{\alpha\to \kappa}\dist(\S_{\Y},Z_{\alpha})=r^{-1}(1-2\epsilon)$ for some $\epsilon>0$. This reads
$\lim_{\alpha\to \kappa}\dist(r\S_{\Y},Z_{\alpha})=1-2\epsilon$. 
Then there is a sequence $\{y_{\alpha}\}_{\alpha<\kappa}\subset r\S_{\Y}$ such that 
$\dist(y_{\alpha},Z_{\alpha})<1-\epsilon$ for sufficiently large ordinals $\alpha<\kappa$.
Since the Lindel\"{o}f number of $r\S_{\Y}$ is less than $\kappa$, we obtain that 
$\bigcap_{\beta<\kappa}\overline{\{y_{\alpha}:\ \beta<\alpha<\kappa\}}\neq \emptyset$ and pick $y$ from this set. 
Observe that  
\[y\in \bigcap_{\alpha<\kappa}\overline{(1-\epsilon)\B_{\X}+Z_{\alpha}}\subset (1-\epsilon)r\B_{\X}.\] 
Thus, we arrive at a contradiction, since $y\in r\S_{\Y}$.
\end{proof}

It is perhaps worthwhile to observe closely in the arguments the intertwined roles of the structure of Banach spaces, the Lindel\"{o}f number of subsets and the length of the chains. 

\subsection{Banach subspaces analogy for comeager sets of a Baire space}
\ \\
Suppose that $\X$ is a Banach space and $\Y\subset \X$ is a closed subspace. We will denote 
$\mathrm{codens}(\Y)=\mathrm{dens}(\X\mod \Y)$, provided that the superspace $\X$ is understood. 
If $\mathrm{codens}(\Y)=\omega$, we shall say that $\Y$ is \emph{coseparable}.
This concept can be motivated by its analogy with comeager sets in topology.
Namely, if the superspace $\X$ is nonseparable, then the coseparable subspaces $\Y$ are fat in a sense.  
Recall that comeager subspaces of a Baire space are preserved in countable intersections, and one might ask if the 
same holds for coseparable subspaces of a nonseparable Banach space. 

For example, if $\X$ is reflexive and $(\Y_{n})$ is a countable sequence of coseparable subspaces, 
then $\bigcap_{n}\Y_{n}$ is coseparable. This can be seen by using well-known properties of reflexive spaces as follows: 
\[\X\Big{/} \bigcap_{n}\Y_{n}=\X^{\ast\ast}\Big{/} \bigcap_{n}\Y_{n}^{\ast\ast}=\left(\overline{\span}\bigcup \Y_{n}^{\bot} \right)^{\ast},\]
where the annihilators $\Y_{n}^{\bot}$ are separable due to coseparability. This argument can be modified to cover the more 
general case, where $\X$ is coseparable in its bidual. 

However, in general Banach spaces coseparability is not necessarily stable even in finite intersections. For example, let $\X=\ell^{2}\oplus \ell^{\infty}$, 
\[\Y_{1}=\{(x_{1},0,2^{-2}x_{2},0,3^{-2}x_{3},0,\ldots)\times (x_{1},x_{2},x_{3},\ldots)\in \X:\ \}\] 
and
\[\Y_{2}=\{(0,x_{1},0,2^{-2}x_{2},0,3^{-2}x_{3},\ldots)\times (x_{1},x_{2},x_{3},\ldots)\in \X:\ \}.\] 
Observe that $(\ell^{2}\oplus \{0\})+\Y_{1}=(\ell^{2}\oplus \{0\})+\Y_{2}=\ell^{2}\oplus \ell^{\infty}$. For this reason 
$\X\mod \Y_{1}$ and $\X\mod \Y_{2}$ are separable by using the continuity of the respective quotient mappings. Note that $\Y_{1}\cap \Y_{2}=\{0\}\subset \X$ and $\codens(\{0\})=2^{\omega}$. Essentially the same argument yields the following observation.
\begin{proposition}
Let $\Z$ be a separable space and $\X$ a space with $\omega^{\ast}$-separable dual.
Then there exist coseparable subspaces $Y_{1},Y_{2}\subset \Z\oplus\X$, whose intersection is $\{0\}$. This trivial intersection is not coseparable if $\X$ is nonseparable. 
In particular, if $K$ is any separable, non-metrizable compactum, then $c_{0}\oplus C(K)$ contains two coseparable subspaces with trivial intersection, which is not coseparable. 
\end{proposition}
\qed

It is not hard to check the following fact by using the Hahn-Banach theorem.
\begin{proposition}
The coseparable subspaces of $\X$ are preserved in countable intersections if and only if the following condition holds: 
\begin{enumerate}
\item[($\sigma$)]{Given any $\omega^{\ast}$-closed, $\omega^{\ast}$-separable subspace $\Z\subset\X^{\ast}$, then 
$\mathrm{codens}(\Z_{\bot})=\omega$.} 
\end{enumerate}
\end{proposition}
\qed

\noindent The condition ($\sigma$) can also be viewed as a reverse Asplund property.

Yet another way to describe this class of Banach spaces would be to say that they are saturated with separable quotients. 
Recall the well-known separable quotient problem, which asks whether each nonseparable Banach space 
has a separable infinite-dimensional quotient. Observe that the finite-codimensional subspaces are a fortiori 
coseparable, and thus the separable quotient problem has a positive answer in the class of Banach spaces $\X$ 
satisfying ($\sigma$). In fact, ($\sigma$) states that this happens in a very strong way, and by using the Hahn-Banach extension of functionals one can see that also every infinite-dimensional subspace of $\X$ has a separable infinite-dimensional quotient.

\section{Main results}

Next, we will give the main result.

\begin{theorem}\label{thm: first_main}
Let $\X$ be a Banach space and $\{x_{\alpha}\}_{\alpha<\kappa}$ be a dispersed sequence of $\X$. 
Then one can extract increasing subsequences $\{\alpha_{\sigma}\}_{\sigma<\kappa}\subset \kappa$ as follows.
\begin{enumerate}
\item[(1)]{Suppose that $\X$ satisfies $(\mathrm{I})$ (resp. $(\mathrm{II})$). Then there is a bounded (resp. $1$-bounded) biorthogonal sequence $\{x_{\alpha_{\sigma}}\}_{\sigma<\kappa}$.}
\item[(2)]{If $\X$ satisfies $(\mathcal{B})$ and $\{x_{\alpha}\}_{\alpha<\kappa}$ is $\SD$, then there is a basic sequence $\{x_{\alpha_{\sigma}}\}_{\sigma<\kappa}$.}
\item[(3)]{If $\Z\subset\X^{\ast}$ is a norming subspace such that $\{x_{\alpha}\}_{\alpha<\kappa}$ is $\sigma(\X,\Z)$-null, then there exists a basic sequence $\{x_{\alpha_{\sigma}}\}_{\sigma<\kappa}$.}
\end{enumerate}
If, additionally, $\X$ satisfies $(1$-$\mathcal{B})$ in $(2)$, or $\Z$ is $1$-norming in $(3)$, then the basic sequence can be chosen to be monotone. 
\end{theorem}

We will give the proof shortly after some remarks. Note that if $\X$ in Theorem \ref{thm: first_main}(2) satisfies condition $(\mathrm{II})$, then $\{x_{\alpha}\}_{\alpha<\kappa}$ is already weakly null. Recall that such a sequence is $\SD$.

Observe that if $\{x_{\alpha_{\sigma}}\}_{\sigma<\kappa}\subset \X$ is not $\SD$ for any cofinal subsequence $\{\alpha_{\sigma}\}_{\sigma<\kappa}\subset\kappa$, then there clearly does not exist
an M-basic subsequence of $\{x_{\alpha}\}_{\alpha<\kappa}$ of length $\kappa$, since M-basic sequences are necessarily $\SD$.
The analogous statement holds for dispersed sequences and biorthogonal sequences.
Observe that each minimal sequence $\{x_{\alpha}\}_{\alpha<\kappa}$ contains a uniformly minimal subsequence of length
$\kappa$, since $\limsup_{\beta\to\kappa}\dist(x_{\beta},[x_{\alpha}:\ \alpha\neq \beta])>0$ by the regularity of $\kappa$. 
Thus we obtain the following dichotomy.

\begin{corollary}
Let $\X$ be a Banach space with property $(C)$ and let $\{x_{\alpha}\}_{\alpha<\kappa}\subset \X$. Then exactly one of the following conditions hold:
\begin{itemize}
\item[$\bullet$]{There is no subsequence $\{x_{\alpha_{\sigma}}\}_{\sigma<\kappa}$, which is dispersed,}
\item[$\bullet$]{There is a $1$-bounded biorthogonal sequence $\{x_{\alpha_{\sigma}}\}_{\sigma<\kappa}$,}
\end{itemize}
and exactly one of the following conditions hold:
\begin{enumerate}
\item[$\bullet$]{There is no subsequence $\{x_{\alpha_{\sigma}}\}_{\sigma<\kappa}$, which is $\SD$,}
\item[$\bullet$]{There is a weakly null monotone basic sequence $\{x_{\alpha_{\sigma}}\}_{\sigma<\kappa}$.}
\end{enumerate}
\end{corollary}
\qed

\begin{proof}[Proof of Theorem \ref{thm: first_main} ]

We will first consider the hardest case (2), where $\X$ satisfies $(\mathcal{B})$. 
For each $\theta<\kappa$ let $\eta(\theta)$ be the infimum of numbers $C\geq 1$ such that there exists 
$\gamma<\kappa$ and a continuous linear projection 
\[P\colon [x_{\alpha}:\ \alpha\in [0,\theta]\cup [\gamma,\kappa)]\to [x_{\alpha}:\ \alpha\in [0,\theta]]\] 
given by $P(x+y)=x$ for $x\in [x_{\alpha}:\ \alpha\in [0,\theta]],\ y\in  [x_{\alpha}:\ \alpha\in  [\gamma,\kappa)]$
with $\|P\|\leq C$ (and $\eta(\theta)=\infty$ if such $P$ does not exist). Let $\epsilon>0$. Suppose that $\theta_{1}\leq\theta_{2}<\kappa$ and 
\[P_{2}\colon [x_{\alpha}:\ \alpha\in [0,\theta_{2}]\cup [\gamma_{2},\kappa)]\to [x_{\alpha}:\ \alpha\in [0,\theta_{2}]]\]  
is an admissible projection in the definition $\eta(\theta_{2})$ with $\|P_{2}\|\leq \eta(\theta_{2})+\epsilon$. Then, putting $P_{1}=P_{2}|_{[x_{\alpha}:\ \alpha\in [0,\theta_{1}]\cup [\gamma_{2},\kappa)]}$
defines a projection, which is admissible in the definition of $\eta(\theta_{1})$ and again $\|P_{1}\|\leq \eta(\theta_{2})+\epsilon$.
We conclude that $\eta\colon [0,\kappa)\to \R\cup\{\infty\}$ is a non-decreasing function.

Next, we will show that $\eta(\theta)<\infty$ for each $\theta<\kappa$ under the hypothesis $(\mathcal{B})$. 
Indeed, it follows from Lemma \ref{lm: angle} that there is $\beta<\kappa$ such that 
the angle between $[x_{\alpha}:\ \alpha\leq \theta]$ and 
$[x_{\alpha}:\ \alpha\geq \beta]$ is strictly positive, which is equivalent to the statement that there is a continuous linear projection 
\[P\colon [x_{\alpha}:\ \alpha\in [0,\theta]\cup [\beta,\kappa)]\to [x_{\alpha}:\ \alpha\leq \theta]\]
given by $P(x+y)=x$ for $x\in [x_{\alpha}:\ \alpha\in [0,\theta]],\ y\in  [x_{\alpha}:\ \alpha\in  [\beta,\kappa)]$.

Thus, the values of $\eta$ are finite. By using the the regularity of $\kappa$ and the fact that $\eta$ is non-decreasing we obtain that $\lim_{\theta\to \kappa}\eta(\theta)$ exists and is finite. Denote this limit by $1\leq C<\infty$.

Let us define an increasing sequence $\Phi\colon [0,\kappa)\to [0,\kappa)$ by letting 
$\Phi(\theta)$ be the least $\phi$ such that there is a projection 
\[P\colon [x_{\alpha}:\ \alpha\in [0,\theta]\cup [\phi,\kappa)]\to [x_{\alpha}:\ \alpha\in [0,\theta]]\] 
with $\|P\|\leq C$. Indeed, this can be accomplished by the regularity of $\kappa$.

The required basic sequence can be extracted by transfinite recursion as follows. Let $\alpha_{0}=0$ and
\[\alpha_{\sigma}=\Phi\left(\sup_{\gamma<\sigma}\alpha_{\gamma}\right)\ \vee\ \left(\sup_{\gamma<\sigma}\alpha_{\gamma}\right)+1,\quad \sigma<\kappa.\]
The relevant basis projections are obtained by restriction from the projections provided by the 
definition of $\Phi$. It is clear that the basis constant is at most $C$.

In proving (1) we apply the quotient mapping $q\colon \X\to \X\mod \bigcap_{\beta<\kappa}[x_{\alpha}:\beta<\alpha<\kappa]$.
Then $\{q(x_{\alpha})\}_{\alpha<\kappa}$ becomes a $\SD$ sequence in 
$\X\mod \bigcap_{\beta<\kappa}[x_{\alpha}:\beta<\alpha<\kappa]$. Since the conditions $(\mathrm{I})$ and 
$(\mathrm{II})$ pass on to quotients and imply $(\mathcal{B})$, according to Proposition \ref{prop1}, 
we may apply (2) to extract a basic sequence $\{q(x_{\alpha_{\beta}})\}_{\beta<\kappa}$. Then, one lifts this sequence to obtain a corresponding biorthogonal system. Indeed, this biorthogonal sequence will be bounded, 
since the basis projections $P_{\lambda}$ on the quotient are uniformly bounded and the suitable biorthogonal functional $f_{\alpha_{\beta}}$ is obtained from the $1$-dimensional projection $P_{\lambda+1}-P_{\lambda}$. 

Next, we will indicate the key modification to obtain the monotonicity of the basis in the case where $\{x_{\alpha}\}_{\alpha<\kappa}$ is $\sigma(\X,\Z)$-null and $\Z\subset\X^{\ast}$ is $1$-norming. 
Given an initial segment $[0,\theta]$ with $\theta<\kappa$ infinite, we let $\{f_{\gamma}\}_{\gamma<\theta}\subset \S_{\X^{\ast}}\cap \Z$ be a set of functionals that $1$-norms $[x_{\alpha}:\ \alpha\leq \theta]$. 
Then, according to the $\sigma(\X,\Z)$-convergence, we find $\beta<\kappa$ such that $f_{\gamma}(x_{\alpha})=0$ for $\beta\leq \alpha<\kappa$ and $\gamma<\theta$.
It follows that the natural linear projection 
\[P\colon [x_{\alpha}:\ \alpha\in [0,\theta]\cup [\beta,\kappa)]\to [x_{\alpha}:\ \alpha\leq \theta]\]
given by $P(x+y)=x$ for $x\in [x_{\alpha}:\ \alpha\in [0,\theta]],\ y\in  [x_{\alpha}:\ \alpha\in  [\beta,\kappa)]$ is contractive, since $\|x\|=\sup_{f\in\Z\cap \S_{\X^{\ast}}}f(x)=\sup_{f\in\Z\cap \S_{\X^{\ast}}}f(x+ty)\leq \|x+ty\|$ for $t\in \R$.
\end{proof}

\begin{remark}
The recursive method in the proof of Theorem \ref{thm: first_main} can be adapted for verifying the following fact:
Suppose that in Theorem \ref{thm: first_main} for each $\gamma<\kappa$ there are ordinals
$\gamma<\mu<\nu<\kappa$ such that $[x_{\mu}]^{\bot}$ is $1$-norming for $[x_{\alpha}:\ \alpha\in [0,\kappa)\setminus (\gamma,\nu)]$. Then, one can extract a suppression unconditional basic sequence of length $\kappa$.
\end{remark}

\begin{theorem}\label{thm: second_main}
Let $\X$ satisfy $(\mathcal{B})$ and admit a generating $\SD$ sequence $\{x_{\alpha}\}_{\alpha<\kappa}$.
\begin{enumerate}
\item[(i)]{If $\kappa=\omega_{1}$, then $\X$ has $\SCP$.} 
\item[(ii)]{If $\kappa<\aleph_{\omega}$ and $\{x_{\alpha}\}_{\alpha<\kappa}$ is an M-basis, then $\X$ has SCP.}
\item[(iii)]{Let $\kappa=\aleph_{k}, 0<k<\omega,$ and $\Y\subset \X$ be a separable subspace. Then there is a subspace $\Z\subset \X$ with $\mathrm{codens}(\Z)\leq\aleph_{k-1}$ and a separable subspace $\mathrm{W}\supset \Y$ complemented in $\Z$.} 
\end{enumerate}
If $\X$ satisfies additionally $(1$-$\mathcal{B})$ in $(\mathrm{i})$ or $(\mathrm{ii})$, then $\X$ has $1$-$\SCP$.
In all cases the complemented subspaces can be chosen in such a way that they are generated by subsequences of $\{x_{\alpha}\}_{\alpha<\kappa}$.
\end{theorem}

\begin{lemma}\label{lm: base_reduction}
Let $\X$ be a Banach space satisfying $(\mathcal{B})$. 
Suppose that $\{x_{\alpha}\}_{\alpha<\kappa^{+}}$ is a $\SD$ sequence of $\X$. Then there exists an ordinal $\phi<\kappa^{+},\ \cf(\phi)=\kappa,$ such that
$\{x_{\alpha}\}_{\alpha<\phi}$ is $\SD$ and $[x_{\alpha}:\ \alpha<\phi]$ is complemented in $[x_{\alpha}:\ \alpha<\kappa^{+}]$.
\end{lemma}
\begin{proof}
By studying the proof of Theorem \ref{thm: first_main} we obtain that there is constant $1\leq C<\infty$ and increasing cofinal sequences $\{\theta_{\beta}\}_{\beta<\kappa^{+}},\{\eta_{\beta}\}_{\beta<\kappa^{+}}\subset \kappa^{+}$ such that 
$\theta_{\beta}<\eta_{\beta}$ and such that the natural projection 
\[P_{\theta_{\beta},\eta_{\beta}}\colon [\{x_{\alpha}:\ \alpha\in [0,\kappa^{+})\setminus (\theta_{\beta}, \eta_{\beta})\}]\to [\{x_{\alpha}:\ \alpha \in [0,\theta_{\beta}]\}]\]
satisfies $\|P_{\theta_{\beta},\eta_{\beta}}\|\leq C$ for each $\beta$.
Since $\{\theta_{\beta}\}_{\beta<\kappa^{+}}$ is cofinal in $\kappa^{+}$, we can define a function
$\Phi\colon [0,\kappa^{+})\to [0,\kappa^{+})$ by letting $\Phi(\alpha)$ be the least ordinal $\sigma<\kappa^{+}$
such that the linear projection $P_{\alpha,\sigma}$ is well-defined and $\|P_{\alpha,\sigma}\|\leq C$ holds.
We define $\phi(0)=0$ and recursively
\[\phi(\beta)=\Phi(\sup_{\alpha<\beta}\phi(\alpha)+1)\]
for $1\leq \beta<\kappa$. Observe that this defines an increasing sequence.
Let $\phi=\sup_{\beta} \phi(\beta)$. Observe that $\cf(\phi)=\kappa$. Moreover, by the construction of $\phi$ 
it holds that $\Phi([0,\phi))\subset [0,\phi)$.

We claim that $\phi$ is the required ordinal and the natural projection
\[P\colon [x_{\alpha}:\ \alpha<\kappa^{+}]\to [x_{\alpha}:\ \alpha < \phi]\]
is well-defined and satisfies $\|P\|\leq C$. Towards this, let $x\in [x_{\alpha}:\ \alpha < \phi]$, $\|x\|=1$.

Fix a countable subset $\Gamma_{0}\subset [0,\phi)$ such that $x\in [x_{\alpha}:\ \alpha \in \Gamma_{0}]$. Since $\Phi(\sup\Gamma_{0})<\phi$, we obtain the following facts by the construction of $\Phi$.
\begin{enumerate}
\item[(i)]{We have $x\notin [x_{\alpha}:\ \Phi(\sup\Gamma_{0})\leq \alpha <\phi]$. We conclude that 
$\{x_{\alpha}\}_{\alpha<\phi}$ is $\SD$.}
\item[(ii)]{We have $\|x+z\|\geq 1/C$ for all $z\in [x_{\alpha}:\ \Phi(\sup\Gamma_{0})\leq \alpha <\kappa^{+\\}]$. In particular, this holds for $z\in [x_{\alpha}:\ \phi\leq \alpha <\kappa^{+\\}]$.
We conclude that $[\{x_{\alpha}\}_{\alpha<\phi}]$ is complemented in $[\{x_{\alpha}\}_{\alpha<\kappa^{+}}]$ in the natural way, the projection constant being at most $C$.}
\end{enumerate}
\end{proof}

\begin{lemma}\label{lm: Mbaselemma}
Let $\{(x_{\alpha},x_{\alpha}^{\ast})\}_{\alpha<\lambda}\subset \X\times \X^{\ast}$ be an M-basis. Then for each uncountable regular cardinal $\kappa<\lambda$ there is a mapping $\phi\colon\kappa\to \lambda$
such that $\{(x_{\phi(\beta)},x_{\phi(\beta)}^{\ast})\}_{\beta<\kappa}$ is an M-basic sequence and $\kappa\subset \phi(\kappa)$.
\end{lemma}
\begin{proof}
Let $\Gamma_{0}=[0,\kappa)$. Since the Lindel\"{o}f number of $\S_{\X}\cap [x_{\alpha}:\ \alpha\in\Gamma_{0}]$ is $\kappa$ and the biorthogonal functionals separate $\X$, 
we obtain that there is a subset $\Gamma_{1}\subset \lambda$ such that $\Gamma_{0}\subset \Gamma_{1}$, $|\Gamma_{1}|=\kappa$ and $\{x_{\alpha}^{\ast}\}_{\alpha\in \Gamma_{1}}$ separates 
$[x_{\alpha}:\ \alpha\in \Gamma_{0}]$. We proceed by a recursion of length $\omega_{1}$ to construct $\Gamma=\bigcup_{\sigma<\omega_{1}}\Gamma_{\sigma}$. Then 
$[x_{\alpha}:\ \alpha\in \Gamma]=\bigcup_{\sigma<\omega_{1}}[x_{\alpha}:\ \alpha\in \Gamma_{\sigma}]$, so that reorganizing $\{\phi(\beta)\}_{\beta<\kappa}=\Gamma$ yields the required mapping.
\end{proof}

\begin{proof}[Proof of Theorem \ref{thm: second_main}]
Suppose that $\Y\subset\X$ is a separable subspace. In verifying SCP we may assume without loss of generality that $\Y\subset [x_{\alpha}:\ \alpha<\omega]$. 

Let us consider the claim (i). By the proof of Theorem \ref{thm: first_main} we obtain a sequence $\{\sigma_{\beta}\}_{\beta<\omega_{1}}\subset \omega_{1}$ such that 
there is a bounded linear projection $P_{\beta,\sigma_{\beta}}\colon [x_{\alpha}:\ \alpha\in [0,\omega_{1})\setminus [\beta,\sigma_{\beta})]\to [x_{\alpha}:\ \alpha<\beta]$ for each $\beta$. 
Then $\liminf_{\beta\to\omega_{1}}\|P_{\beta,\sigma_{\beta}}\|=C<\infty$ according to the 
regularity of $\omega_{1}$. Thus we can find a sequence of indices $\omega<\beta_{0}<\sigma_{0}<\beta_{1}<\sigma_{1}<\ldots,$ such that for all $n<\omega$ there is a linear projection 
$P_{\beta_{n},\sigma_{n}}$ with $\|P_{\beta_{n},\sigma_{n}}\|\leq C$. Then according to the construction of the sequence, the mapping $\span(x_{\alpha}:\ \alpha<\omega_{1})\to \span(x_{\alpha}:\ \alpha<\bigvee_{n}\beta_{n})$ given by 
$x+z\mapsto x$ for $x\in \span(x_{\alpha}:\ \alpha<\bigvee_{n}\beta_{n})$ and $z\in \span(x_{\alpha}:\ \bigvee_{n}\beta_{n}\leq \alpha <\omega_{1})$ defines a bounded linear projection onto $\span(x_{\alpha}:\ \alpha<\bigvee_{n}\beta_{n})$
in the dense linear subspace $\span(x_{\alpha}:\ \alpha<\omega_{1})$ of $\X$. Let us extend this projection to the whole space. Thus we have obtained a continuous linear projection onto
$[x_{\alpha}:\ \alpha<\bigvee_{n}\beta_{n}]$.

Let us consider the claim (iii). Write $\aleph_{k}=\kappa$ for the suitable $k<\omega$. By applying Lemma 
\ref{lm: base_reduction} we obtain that there is a bounded 
linear projection $P_{k-1}$ onto $[x_{\alpha}:\ \alpha<\phi_{k-1}]$ for some $\phi_{k-1}<\kappa,\ \cf(\phi_{k-1})=\aleph_{k-1}$. Pick a cofinal subsequence $\psi_{k-1}=\{\alpha_{\beta}\}_{\beta<\aleph_{k-1}}\supset \omega$ of $\phi_{k-1}$. 
We apply Lemma \ref{lm: base_reduction} again to find that there is $\phi_{k-2}<\aleph_{k-1},\ \cf(\phi_{k-2})=\aleph_{k-2},$ and a bounded linear projection 
$P_{k-2}\colon  [x_{\alpha}:\ \alpha\in \psi_{k-1}] \to [x_{\alpha}:\ \alpha<\phi_{k-2},\ \alpha\in \psi_{k-1}]$. Next, we will pick a further increasing subsequence $\psi_{k-2}\colon \aleph_{k-2}\to \psi_{k-1}$ such that 
$\psi_{k-2}$ is cofinal in $\phi_{k-2}$. We proceed in this manner to produce projections 
$P_{k-i-1}\colon [x_{\alpha}:\ \alpha\in \psi_{k-i}]\to [x_{\alpha}:\ \alpha<\phi_{k-i-1},\ \alpha\in \psi_{k-i}]$ for $1\leq i\leq k-1$.
Write $\phi_{k}=0$ and $\psi_{k}=[0,\kappa)$. Put $\Gamma=\psi_{1}\cup \bigcup_{0\leq i<k}[\phi_{i},\kappa)\cap \psi_{i+1}$ and $\Z=[x_{\alpha}:\ \alpha\in \Gamma]$.

Now, since $k<\omega$, this produces a finite chain of continuous linear projections
\[\Z=\Z_{k}\longrightarrow \Z_{k-1}\longrightarrow \ldots \longrightarrow \Z_{i}\longrightarrow \ldots\longrightarrow \Z_{1}\longrightarrow \Z_{0}\supset \Y,\]
where $\Z_{i}$ has density $\aleph_{i}$ for $0\leq i\leq k$. Observe that $[\phi_{k-1},\kappa)\subset \Gamma$, so that $\mathrm{codens}(\Z)\leq \aleph_{k-1}$.

The claim (ii) is obtained from the proof of claim (iii) as follows. Since M-basic sequences of regular length are $\SD$ under any reordering, we may use the above projections $P_{i}$ with the exception that after each 
step we may reorganize the indices in such a way that $\phi_{i}=\aleph_{i}$. Thus no indices need be discarded and we can set $\Z=\X$.
\end{proof}

Apparently, the requirement of the existence of a $\SD$ sequence $\{x_{\alpha}\}_{\alpha<\kappa}\subset\X$ somewhat narrows down the class of Banach spaces. For example, neither Kunen's $C(K)$ space (\cite{Ne}), Shelah's space (\cite{Sh}), nor any other $\mathrm{CSP}$ space admits such a sequence, see \cite{Gr,Tal}. Every nonseparable dual space 
has an uncountable biorthogonal system (see \cite[Cor. 4]{Ste}) and thus a dispersed sequence but not necessarily a $\SD$ sequence of regular uncountable length. For instance, this is the case with the dual of the James Tree space, $JT^{\ast}$, which has $\mathrm{CSP}$. Also, the $\mathrm{HI}$ spaces which admit a $\SD$ sequence $\{x_{\alpha}\}_{\alpha<\kappa}$ 
cannot satisfy conditions $(\mathrm{I})$, $(\mathrm{II})$, or, a fortiori, property $(C)$.

\begin{theorem}\label{thm: one_norming_M-basis}
Let $\X$ be a Plichko space with a countably $1$-norming M-basis $\{x_{\alpha}\}_{\alpha<\theta}$,
where $\theta$ is an uncountable cardinal. Then there exists an injective mapping $\phi\colon \theta\to \theta$ such that
$\{x_{\phi(\alpha)}\}_{\alpha<\theta}$ forms a monotone basic sequence. Moreover, if $\theta$ is regular, then $\{\phi(\alpha)\}_{\alpha<\theta}$ can be taken to be an increasing sequence.
\end{theorem}
\begin{proof}
According to the countably $1$-norming property of the M-basis, we may choose for each $A\in \mathcal{P}(\theta)$ a set $\Lambda(A) \in \mathcal{P}(\theta)$ such that the resulting mapping $\Lambda\colon 2^{\theta}\to 2^{\theta}$ 
satisfies the following properties:
\begin{enumerate}
\item[(i)]{$A\subset \Lambda(A)$,}
\item[(ii)]{$|\Lambda(A)|\leq |A|\vee \omega$,}
\item[(iii)]{$[x_{\alpha}:\ \alpha\in \theta\setminus \Lambda(A)]^{\bot}$ $1$-norms $[x_{\alpha}:\ \alpha\in A]$.}
\end{enumerate}
Observe that then there exists a contractive projection
\[P\colon [x_{\alpha}:\ \alpha\in A \cup (\theta\setminus \Lambda(A))]\to  [x_{\alpha}:\ \alpha\in A]\]
for each $A\subset \mathcal{P}(\theta)$. Indeed, the functionals in 
$[x_{\alpha}:\ \alpha\in \theta\setminus \Lambda(A)]^{\bot}$ witness the fact that the relevant projections will be contractive similarly as at the end of the proof of Theorem \ref{thm: first_main}. The mapping $\phi$ is constructed 
recursively as follows. Let $\phi(0)=0$. For each $\mu<\theta$ and $\{\phi(\alpha)\}_{\alpha<\mu}\subset \theta$ we pick 
$\phi(\mu)\in \theta\setminus \Lambda(\{ \phi(\alpha):\ \alpha<\mu \})$. Then 
$|\Lambda(\{ \phi(\alpha):\ \alpha\leq \mu \})|\leq |\mu|\vee \omega$.
\end{proof}

In the proof of Theorem \ref{thm: first_main}, say, in the case $\dens(\X)=\omega_{1}$, we applied the well-ordering of the $\SD$ sequence, which, combined with the structural assumption $(\mathcal{B})$ of the space, yields that the sequence splits to three parts: 
the separable part generating the range of the basis projection, the separable residual part, and the nonseparable tail, which yields the kernel of the basis projection. Setting the residual middle part aside, we applied here a phenomenon in Banach spaces, which can be viewed as a weaker form of the $\SCP$ and which is mostly due to the well-ordering of the system of vectors. Namely, the intersections of the tail spaces are coseparable here, and thus we were able to apply 'combinatorial-geometro-topological' argument on the separable initial parts to extract a basic sequence step by step. 
There appear to be some connections between condition $(\mathcal{B})$, the $\SCP$ and the preservation of coseparable subspaces in countable intersections, and these are studied in the next section.

\section{WLD spaces and coseparable subspaces}

First, we would like to pose the following problems: Which 'large-density-related properties' of a Banach space are inherited by the coseparable subspaces? For example, if $\X$ is a non-WCG space then each coseparable subspace of $\X$ 
is such a space, see \cite{Valdivia}. What can be said about properties of Banach spaces holding up to restriction to coseparable subspaces? This question should be compared to Theorem \ref{thm: second_main} (iii) and Theorem \ref{thm: sigma_cosep} (ii).

The following observation is a consequence of WLD space characterization \eqref{eq: count}.
\begin{corollary}
In a WLD space $\X$ let $\Y\subset\X$ be a nonseparable subspace and $\Z\subset\X$ a coseparable subspace. Then $\dens(\Y\cap \Z)=\dens(\Y)$. More generally,
\[\dens(\X)=\dens(\bigcap_{n}E_{n})\ \vee\ \bigvee_{n}\codens(E_{n}),\]
where $(E_{n})$ is any countable sequence of closed subspaces of $\X$.
\end{corollary}
\qed

\begin{theorem}\label{thm: cosep}
Let $\X$ be a Banach space with the property that for any countable sequence of functionals $(f_{n})_{n<\omega}\subset\X^{\ast}$ there exists a generating family of vectors 
$\{x_{\gamma}\}_{\gamma\in \Gamma}\subset\X$ such that $|\{\gamma:\ f_{n}(x_{\gamma})\neq 0\}|\leq \aleph_{0}$ for $n<\omega$. 
Then $(\sigma)$ holds, i.e. the coseparable subspaces of $\X$ are stable in countable intersections. 
In particular, this is the case if $\X$ is WLD.
\end{theorem}
\begin{proof}
Suppose that $\X$ is a Banach space with the above property. We will show that if $(\Z_{n})$ is a sequence of closed subspaces of $\X$ such that $\omega^{\ast}$-$\mathrm{dens}((\X\mod \Z_{n})^{\ast})=\omega$ for $n<\omega$, then 
$\X\mod \bigcap_{n<\omega}\Z_{n}$ is separable. 

Indeed, for each $n$ let $(g_{n,k})_{k}\subset \Z_{n}^{\bot}$ be a sequence which separates $\X\mod \Z_{n}$.
Reorganize $(f_{n})_{n<\omega}=\{g_{n,k}:\ n,k<\omega\}$ and let $\{x_{\gamma}\}_{\gamma\in\Gamma}$ be a generating family of 
vectors as in the assumptions. Write 
\[\Gamma_{0}=\{\gamma\in \Gamma:\ f_{n}(x_{\gamma})\neq 0,\ \mathrm{for\ some}\ n<\omega\},\]
and observe that $\Gamma_{0}$ is countable. 
According to the selection of $(f_{n})$, there is for each $x\in \X\setminus \bigcap_{n<\omega}\Z_{n}$ an index $n$
such that $f_{n}(x)\neq 0$. This reads 
\[ [x_{\gamma}:\ \gamma\in \Gamma\setminus \Gamma_{0}]\subset \bigcap_{n<\omega}\ker f_{n}\subset\bigcap_{n<\omega}\Z_{n}.\] 
Since $\X=[\{x_{\gamma}\}_{\gamma}]$, we have the following chain of surjective quotient mappings
\[ [x_{\gamma}:\ \gamma\in \Gamma_{0}]\longrightarrow \X\mod [x_{\gamma}:\ \gamma\in \Gamma\setminus \Gamma_{0}] \longrightarrow \X\mod \bigcap_{n<\omega}\ker f_{n} \longrightarrow \X\mod \bigcap_{n<\omega}\Z_{n}.\]
As the quotient mappings are continuous (in fact contractive), we obtain that $\X\mod \bigcap_{n<\omega}\Z_{n}$ is separable.
The proof is finished by taking into account the characterization \eqref{eq: count} of WLD spaces.
\end{proof}

\begin{theorem}\label{thm: sigma_cosep}
Suppose that $\X$ is a nonseparable Banach space satisfying $(\sigma)$. 
\begin{enumerate}
\item[(i)]{Then $\X$ has a monotone basic sequence of length $\omega_{1}$. Moreover, any basic sequence of $\X$ having countable order type has an uncountable extension.} 
\item[(ii)]{Given a separable subspace $A\subset \X$ there exists a coseparable subspace $M\subset \X$ such that
$A$ is $1$-complemented in $M$.}
\end{enumerate}
\end{theorem} 
\begin{proof}
Let us check the latter claim in $(i)$. This argument essentially covers both claims. Let $(x_{n})_{n<\alpha},\ \alpha<\omega_{1},$ be a countable basic sequence on $\X$. 
By using the separability of $[x_{n}:\ n<\alpha]$ we may let $(f_{i})_{i<\omega}\subset\X^{\ast}$ be a $1$-norming sequence for $[x_{n}:\ n<\alpha]$.  According to $(\sigma)$ we have that $\bigcap_{i<\omega}\ker(f_{i})$ is a coseparable subspace, 
in particular non-trivial. Hence we may pick $x_{\alpha}\in \bigcap_{i<\omega}\ker(f_{i}),\ \|x_{\alpha}\|=1$. Note that $\|x\|\leq \|x+tx_{\alpha}\|$ for any $x\in [x_{n}:\ n<\alpha]$ and $t\in\R$.
We proceed by recursion of length $\omega_{1}$.

Let us check $(ii)$. Let $(g_{n})_{n<\omega}\subset \S_{X^{\ast}}$ be a $1$-norming sequence for $A$. Let $\Z=\bigcap_{n<\omega}\ker(g_{n})$ and $M=[A\cup \Z]$. According to $(\sigma)$ $\Z$ is coseparable, and thus $M$ is coseparable. Since $(g_{n})_{n<\omega}\subset \S_{X^{\ast}}$ is $1$-norming, we obtain, similarly as in the $\sigma(\X,\Z)$-null case of the proof of Theorem \ref{thm: first_main}, that there is a contractive projection $P\colon M\to A$. 
\end{proof} 

\begin{theorem}
Let $\X$ be a Banach space. Then the following conditions are equivalent:
\begin{enumerate}
\item[(i)]{$\X$ is WLD}
\item[(ii)]{$\X$ has property $(C)$ and admits an M-basis.}
\item[(iii)]{$\X$ satisfies condition $(\mathrm{II})$ and admits an M-basis.}
\end{enumerate}
\end{theorem}
\begin{proof}
The directions (i)$\Leftrightarrow$(ii) are known (see \cite{Hajek_biortsyst}), and (ii)$\Rightarrow$(iii) is clear. 

Thus it suffices to check that the implication (iii)$\implies$(i) holds. Suppose that $\{(f_{\alpha},x_{\alpha})\}_{\alpha}$ is an M-basis of $\X$. 
We will check that the characterization \eqref{eq: count} for WLD spaces holds. 

Assume to the contrary that there is $f\in \X^{\ast}$ such that $\Gamma=\{\alpha :\ f(x_{\alpha})\neq 0\}$ 
satisfies $|\Gamma|\geq \omega_{1}$. By applying Lemma \ref{lm: Mbaselemma}, we find an M-basic sequence $\{x_{\alpha_{\beta}}\}_{\beta<\omega_{1}}\subset \{x_{\alpha}\}_{\alpha}$ containing $\omega_{1}$-many indices of $\Gamma$. 

Put $\Y=[\{x_{\alpha_{\beta}}\}_{\beta<\omega_{1}}]$. Since the chosen M-basic sequence is total on $\Y$, we have that 
\[\bigcap_{\gamma<\omega_{1}}\bigcap_{0\leq\beta\leq\gamma}\ker (f_{\alpha_{\beta}}|_{\Y})=\{0\}\subset\Y,\]
so that condition $(\mathrm{II})$ can be applied. We obtain that there is $\gamma<\omega_{1}$ such that 
$f|_{\Y}\in \left(\bigcap_{0\leq\beta\leq \gamma}\ker(f_{\alpha_{\beta}}|_{\Y})\right)^{\bot}\subset\Y^{\ast}$. This contradicts the fact that the chosen M-basic sequence contains uncountably many indices of $\Gamma$. Consequently, \eqref{eq: count} holds.
\end{proof}

\section{Interpreting the conditions on subspace chains in terms of inverse limits} 

\normalfont

The nested sequence of subspaces $\{Z_{\alpha}\}_{\alpha<\kappa}$ of $\X$ yields in a natural way a direct system with the initial object $Z=\bigcap_{\alpha<\kappa}Z_{\alpha}$ and the terminal object $Z_{0}$. 
Next, we will discuss the connections of conditions $(\mathrm{I}),(\mathrm{II})$ and $(\mathcal{B})$ to 
direct and inverse limits. It turns out that there is a very natural reformulation for these conditions in terms of these limits, and this partly motivates the application of the conditions.

Let us begin by observing that $(\mathrm{II})$ has a very simple reformulation in terms of direct limits, namely that
the direct system 
\[Z_{0}^{\bot}\longrightarrow Z_{1}^{\bot}\longrightarrow \ldots \longrightarrow Z_{\alpha}^{\bot}\longrightarrow \ldots, \quad \alpha<\kappa,\]
satisfies $\underset{\longrightarrow}{\lim}\ Z_{\alpha}^{\bot}=Z^{\bot}$. Here the inclusion maps serve as the binding maps.
In a sense, this is a statement about the continuity of the functor $(\cdot)^{\bot}$.

However, the inverse limits are more interesting in this setting. Note that a nested sequence $\{A_{\alpha}\}_{\alpha<\kappa}$ of closed affine subspaces of $\X$ can be written as $\{x_{\alpha}+Z_{\alpha}\}_{\alpha<\kappa}$, where $x_{\alpha}-x_{\beta}\in Z_{\alpha}$ for $\alpha<\beta<\kappa$. Thus each such sequence $\{A_{\alpha}\}_{\alpha<\kappa}$ is realized as an element of the inverse limit of the inverse system
\[\X\mod Z_{0}\longleftarrow \X\mod Z_{1}\longleftarrow \ldots \longleftarrow \X\mod Z_{\alpha}\longleftarrow \ldots,\quad \alpha<\kappa,\]
and vice versa. Here the binding maps are the natural quotient maps $T_{\alpha}^{\beta}\colon \X\mod Z_{\beta}\longrightarrow \X\mod Z_{\alpha},\ x+Z_{\beta}\mapsto x+Z_{\alpha}$ for $\alpha\leq \beta<\kappa$. We endow 
$\underset{\longleftarrow}{\lim}\ \X\mod Z_{\alpha}$ with the norm 
$\|(A_{\alpha})_{\alpha<\kappa}\|_{\underset{\longleftarrow}{\lim}\ \X\mod Z_{\alpha}}=\sup_{\alpha<\kappa}\|A_{\alpha}\|_{\X\mod Z_{\alpha}}=\sup_{\alpha<\kappa}\dist(0,A_{\alpha})$.
It is not hard to check that this norm is complete by using the completeness of the quotient norms and the regularity of $\kappa$.
The natural inclusion mapping $\phi\colon \X\mod Z\hookrightarrow \underset{\longleftarrow}{\lim}\ \X\mod Z_{\alpha}$
given by $x+Z \mapsto (x+Z_{\alpha})_{\alpha<\kappa}$ is an injective, contractive homomorphism. 
Injectivity follows from the fact that for each $x+Z\in \X\mod Z,\ x\notin Z,$ there is $\alpha<\kappa$ 
such that $x\notin Z_{\alpha}$. 

Now, condition $(\mathrm{I})$ states that for each pair of sequences $\{Z_{\alpha}\}_{\alpha<\kappa}$, $\{A_{\alpha}\}_{\alpha<\kappa}$, as above, the corresponding mapping $\phi\colon \X\mod Z\to \underset{\longleftarrow}{\lim}\ \X\mod Z_{\alpha}$ is onto. We note that the open mapping principle yields that $\phi$ is actually an isomorphism in such a case.
This reformulation should be compared to the definition of reflexivity of Banach spaces. 

Condition $(\mathcal{B})$ is equivalent to the statement that for any sequence $\{Z_{\alpha}\}_{\alpha<\kappa}$ having trivial intersection there exists $r\geq 1$ such that $\|x\|\leq r\sup_{\alpha}\dist(x,Z_{\alpha})$ for all $x\in \X$. This holds if and only if $\phi\colon \X\to \underset{\longleftarrow}{\lim}\ \X\mod Z_{\alpha}$ is an embedding. Thus, at this point it becomes apparent why $(\mathrm{I})$ implies $(\mathcal{B})$. Namely, if $\phi$ is an isomorphism, then it is a fortiori an embedding. Observe that $(1$-$\mathcal{B})$ holds if and only if $\phi$ is an isometric embedding. To conclude, if $\X$ satisfies Corson's property $(C)$, then $\phi$ is an isometric isomorphism for each inverse limit as above.

The inverse limits defined by using \emph{countable} sequences of subspaces $(Z_{n})_{n<\omega}$ seem to expand the 
space easily, compared to sequences indexed by $\kappa$, unless we make some strong assumptions about $\X$, like 
reflexivity or the RNP (see \cite{CK}). For example, 
$\underset{\longleftarrow}{\lim}\ (c_{0}\mod [e_{n}:\ n\geq i])_{i<\omega}=\ell^{\infty}$. 

Finally, we will apply the mapping $\phi$ in the setting of topological vector spaces. It turns out that 
considering $\phi$ is useful already in the case where $\phi$ is a closed mapping.
 
Given a topological vector space $Y$ and a sequence $\{Z_{\alpha}\}_{\alpha<\kappa}$ of closed subspaces, the inverse limit of the inverse system
\[Y\mod Z_{0}\longleftarrow Y\mod Z_{1}\longleftarrow \ldots \longleftarrow Y\mod Z_{\alpha}\longleftarrow \ldots,\quad 
\alpha<\kappa\]
is topologized by the product topology inherited from $\prod_{\alpha<\kappa}Y\mod Z_{\alpha}$. 
Recall that the quotient topology on $Y\mod M$, $M\subset \Y$ being a closed linear subspace, consists of open sets of the form $U+M$, where $U\subset Y$ is open. Thus $\overline{A\mod M}\subset Y\mod M$ is $\overline{A+M}\subset Y$ as a set for any $A\subset Y$. Recall that $Y\mod M$ are Hausdorff and even completely regular spaces. The mapping 
$\phi\colon Y\mod Z \to \underset{\longleftarrow}{\lim}\ Y\mod Z_{\alpha}$ is defined similarly as above, and it is 
continuous and injective.

\begin{theorem}\label{thm: tvs_biort}
Let $X$ be a topological vector space and let $\{x_{\alpha}\}_{\alpha<\kappa}\subset X$ be a sequence such that 
\[Y=[x_{\alpha}:\ \alpha<\kappa]\Big{/} \bigcap_{\beta<\kappa}[x_{\alpha}:\ \beta<\alpha<\kappa]\]
has density $\kappa$ and suppose that the canonical inclusion
\[\phi\colon Y\hookrightarrow \underset{\longleftarrow}{\lim}\ (Y\mod [x_{\alpha}:\ \alpha\geq \beta])_{\beta<\kappa}\]
is a closed mapping. Then there exists a minimal sequence $\{x_{\alpha_{\beta}}\}_{\beta<\kappa}$.
\end{theorem}
\begin{proof}
Write $Z_{\beta}=[x_{\alpha}:\ \beta<\alpha<\kappa]$ for $\beta<\kappa$ and $Z=\bigcap_{\beta<\kappa}Z_{\beta}$. 
Without loss of generality we may assume, possibly by deleting suitable indices and selecting a subsequence, that $x_{\alpha}\notin Z$ for $\alpha<\kappa$.
Let us consider the inverse system
\[\{0\}\longleftarrow Y\mod Z_{0}\longleftarrow Y\mod Z_{1}\longleftarrow Y\mod Z_{2}\longleftarrow \ldots \longleftarrow Y\mod Z_{\alpha}\longleftarrow \ldots\quad \alpha<\kappa.\]

Next, we will extract a minimal sequence $\{x_{\gamma_{\mu}}\}_{\mu<\kappa}$. 
We will define increasing sequences $\{\gamma_{\beta}\}_{\beta<\kappa},\{\alpha_{\beta}\}_{\beta<\kappa}\subset \kappa$ recursively as follows. Put $\gamma_{0}=0$. Then there is $\alpha_{0}<\kappa$ such that $x_{0}\notin Z_{\alpha_{0}}$. Since $Y$ has density $\kappa$, there exists $\gamma_{1}\in (\alpha_{0},\kappa)$ such that $x_{\gamma_{1}}\notin \overline{[x_{0}]+Z}$. Indeed, otherwise $\overline{[x_{\alpha}:\ \alpha_{0}<\alpha<\kappa]\mod Z}\subset \overline{[x_{0}]\mod Z}\subset Y$ 
yielding that $Y$ has density less than $\kappa$, a contradiction. Then there exists $\alpha_{1}\in (\alpha_{0},\kappa)$ such that $x_{\gamma_{1}}\notin Z_{\alpha_{1}}$. 

Given $\mu<\kappa$ and $\{\alpha_{\beta}\}_{\beta<\mu},\{\gamma_{\beta}\}_{\beta<\mu}\subset \kappa$ there exists, according to the density assumption on $Y$, an ordinal $\gamma_{\mu}\in (\bigvee_{\beta<\mu}(\alpha_{\beta}\vee \gamma_{\beta}),\kappa)$
such that $x_{\gamma_{\mu}}+Z\notin \overline{[x_{\gamma_{\beta}}:\ \beta<\mu]\mod Z}\subset Y$. Then we choose $\alpha_{\mu}\in (\bigvee_{\beta<\mu}(\alpha_{\beta} \vee \gamma_{\beta}),\kappa)$ such that 
$x_{\gamma_{\mu}}\notin \overline{[x_{\gamma_{\beta}}:\ \beta<\mu]\mod Z_{\alpha_{\mu}}}\subset Y\mod Z_{\alpha_{\mu}}$. 
Indeed, this can be accomplished, since
$\phi\colon Y\to \underset{\longleftarrow}{\lim}\ Y\mod Z_{\alpha}$ is a closed mapping. Namely, $\phi$ maps 
$\overline{[x_{\gamma_{\beta}}:\ \beta<\mu]\mod Z}\subset Y$ to a closed set of $\underset{\longleftarrow}{\lim}\ Y\mod Z_{\alpha}$. Since $\phi$ is an injection, there is an open neighbourhood for $\phi(x_{\gamma_{\mu}}+Z)$, which does not
intersect $\phi(\overline{[x_{\gamma_{\beta}}:\ \beta<\mu]\mod Z})=\overline{\phi([x_{\gamma_{\beta}}:\ \beta<\mu]\mod Z)}$. Since the topology of $\underset{\longleftarrow}{\lim}\ Y\mod Z_{\alpha}$ has a basis consisting of sets of the form $q_{\theta}^{-1}(U)$, where $q_{\theta}\colon \underset{\longleftarrow}{\lim}\ Y\mod Z_{\alpha}\to Y\mod Z_{\theta}$ is the natural quotient map and $U\subset Y\mod Z_{\theta}$ is open, we conclude that there is $\eta<\kappa$ such that 
$x_{\gamma_{\mu}}+Z_{\alpha}\notin \overline{[x_{\gamma_{\beta}}:\ \beta<\mu]\mod Z_{\alpha}}\subset Y\mod Z_{\alpha}$
holds for all $\alpha>\eta$. 

This results in the required minimal sequence $\{x_{\gamma_{\mu}}\}_{\mu<\kappa}$, since 
$x_{\gamma_{\beta}}\in Z_{\alpha_{\mu}}$ for $\beta>\mu$ by the construction.
\end{proof}

\end{document}